\documentclass[12pt]{amsart}
\usepackage{amscd,graphicx,amssymb}
\usepackage[utf8]{inputenc}
\usepackage{amsmath}
\usepackage{amsfonts}
\usepackage{amssymb}
\usepackage{graphicx}
\usepackage{amssymb}
\usepackage{mathtools}
\usepackage{enumitem}
\usepackage{cancel}
\usepackage{array}
\usepackage{lmodern}
\usepackage{microtype}
\usepackage{graphicx}
\usepackage{framed}
\usepackage{color}

\newtheorem{theorem}{Theorem}[section]

\newtheorem{lemma}[theorem]{Lemma}
\theoremstyle{definition}
\newtheorem{definition}[theorem]{Definition}

\makeatletter 
\def\subsubsection{\@startsection{subsubsection}{3}%
	\z@{.5\linespacing\@plus.7\linespacing}{-.5em}%
	{\normalfont\bfseries}}
\makeatother

\author{Peter Berkics}
\title{On self-adjoint linear relations}

\address{University of Pécs, Hungary\\
	H-7633 Pécs, Ifjúság útja 6.} \email{berkicsp@gamma.ttk.pte.hu}
\subjclass[2010]{47A05,47A06,47B25,47B65}
\keywords{Linear operator, Graph of operator, Linear relation, Adjoint of a linear relation, Multi-valued linear operator, von Neumann Theorem }

\begin{document}
	\maketitle

	\begin{abstract}
		
		A linear operator on a Hilbert space $\mathbb{H}$, in the classical approach of von Neumann,  must be symmetric to guarantee self-adjointness. However, it can be shown that the symmetry could be ommited by using a criterion for the graph of the operator and the adjoint of the graph. Namely, $S$ is shown to be densely defined and closed if and only if $\{k+l:\{k,l\}\in G(S)\cap G(S)^*  \}=\mathbb{H}$.

		In a more general setup, we can consider relations instead of operators and we prove that in this situation a similar result holds. We give a necessary and sufficient condition for a linear relation to be densely defined and self-adjoint.

	\end{abstract}

	\section{Inroduction}
	\pagenumbering{arabic}
	\setcounter{page}{1}
	
	Self-adjoint operators on a Hilbert space are an old research topic and very important in many applications. There are well-known old results about them. Recently it turned out that classical theorems can be generalized and even reversed. The new results are quite surprising and we will state them. In this paper we give a generalization of the Stone lemma and another equivalent properties for self-adjointness of operators.

	One can see that this machinery can be treated in the more general language of linear relations. A natural question is that what properties are valid and how we can state the classical theorems in this new language. 
	
	Relations are quite basic notion and well studied in mathematics. They have a crucial role in many topics. It turns out that multi-valued operators can be represented as linear relations. The research of linear relations, as generalization of operators, was started by Richard Arens in 1961 \cite{arens}, who laid the foundation of the operational calculus of linear relations. It is nowadays an active research topic. There are recent generalizations of classical theorems like the von Neumann theorem which states that for a linear operator $T$, the operator $T^*T$ is self-adjoint (see in \cite{sandovici}) and more and more basic results of functional analysis are formulated on linear relation language.\\
	
	We can generalize these results for linear relations, as all the necessary tools make sense for relations. This point of view allows us to apply basic operator properties for even more branches of mathematics. \\
	 
	The generalization of operator theory results on linear relation language has a fairly young history. In the late seventies T. Kato ( \cite{kato1 }, \cite{kato2}) considered unbounded self-adjoint extensions of the sum of two self-adjoint operators in a Hilbert space. To treat this problem, he introduced the concept of form sum. Also, the concept of linear relation arises in the investigation of self-adjoint extensions of the sum of two unbounded self-adjoint operators in a Hilbert space $\mathbb{H}$. The sum is not necessarily a densely defined nonnegative operator on the intersection of the domains, and even the closure of the sum is not necessarily an operator. On the other hand, the sum operator is nonnegative and it has a self-adjoint extension in $\mathbb{H}$ in the sence of multivalued operators i. e. linear relations. \\
	
	After this, a new active research period started in 2006, and since then some important results appeared on generalization of classical operator theory theorems to linear relation language.  \\
	
	Our goal in this paper is to concentrate on some statements on self-adjoint operators. We give a more general Stone lemma and formulate an equivalent condition for self-adjointness of a linear operator on a Hilbert space. We also give the generalization of these theorems for linear relations.\\
	
	The structure of the paper is the following. In Section 2 we remind the classical Stone lemma and von Neumann Theorem and give the reversed von Neumann Theorem and a generalization of the classical Theorem. Theorem \ref{new1} and \ref{new2} are new results.
	In Section 3 we give an introduction to the notion  of linear relations and its properties, and generalize the main Theorems of Section 2 for linear relations. Here Theorem \ref{new3} and \ref{new4} are new. \\
	
	\textit{Acknowledgement} We are grateful for Zoltán Sebestyén for his great research seminars and suggestions.

	\section{Classical Stone lemma and Von Neumann theorem and their generalizations}

	In this section we recall the classical Stone lemma and von Neumann Theorem about self-adjoint operators. Both of them have generalizations and the von Neumann Theorem even has a reverse statement. In the last part we give an equivalent condition for self-adjointness without the explicit statement of the symmetry. \\
	
	First let us recall the definition of the graph of an operator. Let $\mathbb{H}$ a Hilbert space, $T$ a linear operator on $\mathbb{H}$.
	
	\begin{definition}
		We define the \textit{graph} of the operator $T$ as: $$G(T)=\{\{f,Tf\}:f\in dom(T)  \}\text{,}$$
		
		and the adjoint of the graph as:
		$$G(T)^*=\{\{f,f'\}: \{f,f'\}\in \mathbb{H}\times\mathbb{H}, \langle f'h\rangle = \langle f,Th\rangle \text{ for all }\{h,Th\}\in G(T) \}\text{.}$$
		
		Obviously, if $T^*$ exists, then $G(T^*)=G(T)^*$.
		
	\end{definition}
	
	\begin{lemma}
		(Stone lemma) Let $S:\mathbb{H}\to\mathbb{H}$ a symmetric linear operator (i. e. $G(S)\subset G(S^*)$). If $S$ has  full range (i. e. $ran(S)=\mathbb{H}$) then the adjoint $S^*$ exists and $S=S^*$. (see \cite{stone})
	\end{lemma}
	\begin{proof}
		Let $g\in (dom(S))^\bot$ such that $g=Su$ and $u=Sv$. Then 
		$$\langle u,u\rangle = \langle u,Sv\rangle = \langle Su,v\rangle = \langle g,v\rangle = 0$$
		so $u=0$ and $g=0$. Therefore $S$ is densely defined and there exists the adjoint $S^*$ and it extends $S$ i. e. $S\subset S^*$.\\
		
		Let $g\in dom(S^*)$ such that $S^*g=Sv$ and $Sv=g-u$. Then 
		\begin{align*}
			\langle g-u,g-u\rangle &= \langle g-u,Sv\rangle = \langle g,Sv\rangle -\langle u,Sv\rangle  \\
			&=\langle S^*g,v\rangle -\langle Su,v\rangle = \langle Su,v\rangle -\langle Su,v\rangle = 0
		\end{align*}
		
		Then $g-u=0$ so $g=u$ and $g\in dom(S)$.\\
		
		As a consequence $S$ is self-adjoint.
	\end{proof}

	We can generalise the Stone lemma without the explicit useage of symmetry as follows:
	
	\begin{theorem}
		\label{new1}
		Let $S:\mathbb{H}\to\mathbb{H}$ be a linear operator such that $ran(G(S)\cap G(S)^*)=\mathbb{H}$. Then the adjoint $S^*$ exists and $S=S^*$.
	\end{theorem}
	
	\begin{proof}
		Let $T$ be an appropriate operator such that $G(T)=G(S)\cap G(S)^*$. Then $T\subset S$ and $S^*\subset T^*$. \\
		
		If $h,g\in dom(T)$, then $\{g,Tg\}$ and $\{h,Th\}\in (G(S)\cap G(S)^*)$. Therefore $Tg=Sg$, $\{g,Sg\}\in G(S)^*$, $Th=Sh$ and $\{h,Sh\}\in G(S)^*$.\\
		
		Now we show that $T$ is symmetric: 
		$$\langle Tg,h\rangle = \langle Sg,h\rangle = \langle g,Sh\rangle =\langle g,Th\rangle \text{.} $$
		
		Since $T$ is symmetric and has full range, we can apply the Stone lemma for the operator $T$ and we get that $T$ is self-adjoint. As a consequence, $S^*$ exists and $S=S^*$.
	\end{proof}
	
	\begin{theorem} (Classical von Neumann Theorem \cite{NeumannAlapCikk})
		If $T$ is densely defined and closed operator on $\mathbb{H}$, then $T^*T$ is self-adjoint operator on $\mathbb{H}$.
	\end{theorem}
	
	\begin{theorem}
		(Reversed von Neumann Theorem)\cite{reversedVonNeumann} Let $\mathbb{H}$ and $\mathbb{K}$ be both real or complex Hilbert spaces and $T:\mathbb{H}\to\mathbb{K}$  a densely defined linear operator. Assume that the operators $T^*T$ and $TT^*$ on the corresponding Hilbert spaces $\mathbb{H}$ and $\mathbb{K}$, respectively, are self-adjoint. Then $T$ is necessarily closed. 
	\end{theorem}
	
	\begin{theorem}
		(Revised (more general) von Neumann theorem) \cite{reversedVonNeumann} Let $T$ be a not necessarily densely defined symmetric linear operator on a Hilbert space $\mathbb{H}$ with full range i. e. $ran(T)=\mathbb{H}$. Then $T$ is automatically self-adjoint (and clearly densely defined as well). 
	\end{theorem}
	
	The next Theorem characterizes a densely defined operator with the intersection of its graph and the adjoint of its graph.

	\begin{theorem}
		\label{new2}
		Let $S$ be a positive linear operator on the Hilbert space  $\mathbb{H}$. Then the following are equivalent:
		\begin{enumerate}
			\item \label{op1}$dom(S)^\perp={0}$ and $S^* =S$
			\item \label{op2}$\{k+l:\{k,l \}\in G(S)\cap G(S)^*    \}=\mathbb{H}$ 
		\end{enumerate} 
	\end{theorem}
	\begin{proof}
		\text{ }\\
		
		(\ref{op1})$\to $(\ref{op2}):
		
		$\{ k+l:\{ k+l\}\in G(S)  \}=\{k+Sk:k\in dom(S)\}=ran(I+S)=\mathbb{H}$.\\
		
		(\ref{op2})$\to $(\ref{op1}):\\
		
	    We will prove this direction in 3 steps.
		\begin{enumerate}[label=(\alph*)]
			\item $g\in dom(S)^\perp$:\\
			
			In this case we know that  $\{0,g\}\in G(S^*)$ and there exists $\{k,l\}\in G(S)\cap G(S^*)$ such that $g=k+l$. Therefore $l=Sk$, $g=k+Sk$ and  $\langle Sh,k\rangle =\langle h,l\rangle$ for all $h\in dom(S)$. As a consequence, 
			$$0=\langle g,k\rangle = \langle k+Sk,k\rangle = \langle k,k\rangle + \langle Sk, k\rangle \text{.}$$
			Moreover  $\langle Sk,k\rangle \geq 0$ and $\langle k,k\rangle \geq 0$, so $k=0$ and $g=0$. 
			
			Finally we get that $dom(S)^\perp=\{0\}$ and there exists $S^*$ and  $G(S)^*=G(S^*)$.\\
			
			\item $g\in dom(S^*)$:\\
			
			In this case there exists $\{k,l\}\in G(S)\cap G(S^*)$ such that $g+S^*g = k + l$, where $l=Sk=S^*k$, therefore $g+S^*g=k+S^*k$. As a consequence we have  $(g-k)+ S^*(g-k)=0$, so $(g-k)\in ker(I+S^*)=(ker(I+S))^\perp = (ran(I+S))^\perp = \{0\}$. 
			
			Finally $g=k$, $g+S^*g=k+Sk=g+Sg$, that is $S^*g=Sg$, so $S^*\subseteq S$.\\
			
			\item $g\in dom(S)$:\\
			
			In this case there exists $\{k,l\}\in G(S)\cap G(S^*)$ such that $g+Sg=k+l$, where $l=Sk=S^*k$, therefore $g+Sg=k+Sk$. As a consequence, $(g-k)+S(g-k)=0$, so $(g-k)\in ker(I+S)$ and 
			\begin{align*}
				&\langle (g-k)+S(g-k),g-k\rangle =\\
				&= \langle g-k,g-k\rangle + \langle S(g-k),g-k\rangle =0\text{.}
			\end{align*}
			
			We also know the fact that $\langle g-k,g-k\rangle \geq 0$ and $\langle S(g-k),g-k\rangle \geq 0$. Hence $g-k=0$ so $g=k$ and $g+Sg=k+S^*k=g+S^*g$, that is $Sg=S^*g$ thus $S\subseteq S^*$.\\
			
		\end{enumerate}
		
		In summary, with the help of (a), (b) and (c) we proved that the second condition implies that S is densely defined and self-adjoint.
				
		As the second condition is the trivial consequence of the first one, we proved that the two conditions are equivalent.
	\end{proof}
	
	\section{Linear relations}\label{relaciok}
	
	In this section we give a short introduction to the topic of linear relations. Linear relations are natural representational forms for multivalued operators. First let us show the classical definition of linear relations on a Hilbert space.
	
	\begin{definition}
		Let $\mathbb{H}$ be a Hilbert space. We call $T$ a linear relation on $\mathbb{H}$ if  $T$ is a linear subspace of $\mathbb{H}\times\mathbb{H}$. 
		
		Define the domain, range, kernel and multi-value of $T$ as follows:
		\begin{itemize}
			\item $dom(T)=\{f\in \mathbb{H}:\{f,f'\}\in T  \}$,
			\item $ran(T)=\{f'\in \mathbb{H}:\{f,f'\}\in T  \}$,
			\item $ker(T)=\{f\in \mathbb{H}:\{f,0\}\in T  \}$,
			\item $mul(T)=\{f'\in \mathbb{H}:\{0,f'\}\in T  \}$.
		\end{itemize}
	\end{definition}
	
	The previous definition gives the usual characteristics of the linear relation $T$. However, the next definition treats $T$ from a different point of view.

	\begin{definition}
		Let $\mathbb{H}$ be a Hilbert space,  $T$ a linear relation on $\mathbb{H}$ and $x\in\mathbb{H}$. Then  $T(x)=\{y\in \mathbb{H}:  \{x,y \}\in T \}$. 
		In this case the domain, range, and multi-value of  $T$ can be defined as follows:
		\begin{itemize}
			\item $dom(T)=\{x\in\mathbb{H}:T(x) \text{ is not empty }  \}$,
			\item $ran(T)=\bigcup\limits_{x\in dom(T)} T(x)$,
			\item $mul(T)=T(0)$.
		\end{itemize}
	\end{definition}
	
	If $T(x)$ never contains more than one element then we call it single valued. In this case $T$ is the graph of an operator. \\
	
	The linearity of $T$ can be characterized by the following expression:
	$$\alpha T(x_1)+\beta T(x_2)\subseteq T(\alpha x_1 + \beta x_2) \text{\vspace{0.5cm}} \text{ for all } \alpha,\beta \text{ scalar and }x_1,x_2\in \mathbb{H}$$
	
	\begin{definition}
		Let $T$ and $S$ be two linear relations on $\mathbb{H}$. The sum of  $T$ and  $S$ is also a linear relation on $\mathbb{H}$:
		$$T+S=\{ \{ k,l+m \}\in\mathbb{H}\times \mathbb{H}: \{k,l\}\in T, \{k,m\}\in S  \}\text{.}$$
	\end{definition}

	Let us define the product of two linear relations as follows:
	
	\begin{definition}
		Let $T$ and $S$ be two linear relations $\mathbb{H}$. Then the product of the linear relations $S$ and $T$ is also a linear relation on $\mathbb{H}$:
		$$\{ \{k,l\}\in \mathbb{H}\times\mathbb{H}: \{k,x\}\in T, \{x,l\}\in S \text{ for some }x\in\mathbb{H}   \}\text{.}$$
	\end{definition}

	Now we give the definition of the adjoint of a linear relation. 
	
	\begin{definition}
		Let $\mathbb{H}$ be a Hilbert space and $T$ a linear relation on  $\mathbb{H}$. Then the adjoint of $T$ denoted by $T^*$ is also a linear relation on $\mathbb{H}$ and is defined as follows:
		$$T^*=\{ \{f,f'\}\in \mathbb{H}\times\mathbb{H}: \langle f',h\rangle=\langle f,h'\rangle \text{ }\forall \{h,h'\}\in T    \}\text{.}$$
	\end{definition}
	
	It is important to mention that if $T$ is an operator, we can only define the adjoint if it is densely defined. However, if $T$ is a linear relation, we can always define the adjoint.
	
	The notion of self-adjointness and positivity is nearly the same as in the operator case, however the definition of symmetry is a bit different. 
	
	\begin{definition}
		Let $\mathbb{H}$ be a Hilbert space and $T$ be a linear relation on $\mathbb{H}$. Then:
		\begin{enumerate}
			\item We call $T$ \textit{self-adjoint} if $T=T^*$.
			\item We call  $T$ \textit{positive} if $\langle f',f\rangle \geq 0$ for all $\{f,f'\}\in T$. 
			\item We call $T$ \textit{symmetric} if $T\subseteq T^*$.
		\end{enumerate}
	\end{definition}

	\begin{definition}
		Let $\mathbb{H}$ be a Hilbert space and $T$ a linear relation on $\mathbb{H}$. Let the inner product be linear in the first variable and conjugate linear in the second entry. (i. e. $\langle \lambda x, y\rangle= \lambda\langle x,y\rangle $ és  $\langle x,\lambda y\rangle= \overline{\lambda}\langle x,y\rangle $).
		
		Then we can  define the \textit{inner product} on $T$ with the use of inner product on $\mathbb{H}$:
		$$\langle\{k,l\},\{x,y\}\rangle =\langle k,l\rangle + \langle x,y\rangle \text{ \hspace{1cm}}\{k,l\},\{x,y\}\in T\text{.}$$
	\end{definition}
	
	\begin{definition}
		The \textit{inverse} of the linear relation $T$ can be defined as follows:
		$$T^{-1}\{ \{y,x\}\in \mathbb{H}\times \mathbb{H}: \{x,y\}\in T  \}\text{.}$$
	\end{definition}
	
	\begin{definition}
		We call the linear relation $T$ \textit{closed} if $T$ is a closed subspace of $\mathbb{H}\times\mathbb{H}$.
	\end{definition}
	
	\begin{theorem}
		Let $T$ be a positive self-adjoint linear relation. Then $ran(I+T)=\mathbb{H}$.
	\end{theorem}
	 This statement appeared in \cite{sandovici} without proof. The proof is new.
	\begin{proof}
		\text{ }\\
		
		Let us show that the following holds: $$\{0\}=ker(I+T)=ker(T+T^*)=ker\left[(I+T)^*\right]=\left[ran(I+T)\right]^\bot$$
		
		Consider the sequences $(f_n)$ and $(f_n')$ such that $\{f_n,f_n'\}\in T$ and $ (f_n+f_n')\to g\in \mathbb{H}$. 
		
		Then $\{f_n,f_n+f_n'\}\in I+T$ and
		
		$$\langle f_n,f_n+f_n' \rangle = \langle f_n,f_n \rangle + \langle f_n,f_n'\rangle  \geq \langle f_n,f_n\rangle $$\\
		
		Since $T$ is positive,
		
		\begin{align*}
			\langle f_n+f_n', f_n+f_n' \rangle &= \langle f_n,f_n\rangle + \langle f_n,f_n'\rangle + \langle f_n',f_n\rangle + \langle f_n',f_n'\rangle \\
			&\geq \langle f_n,f_n\rangle + \langle f_n',f_n'\rangle \geq \langle f_n,f_n\rangle
		\end{align*} 
		
		Let $f_m$ and $f_m'$ be other elements of the sequence above. Then
		
		$$\langle f_m-f_n,f_m-f_n \rangle \leq \left\langle (f_m+f_m')-(f_m'+f_n),(f_m+f_n')-(f_n+f_m')\right\rangle $$

		The right hand side of this inequality tends to zero, as $m$ and $n$ tends to infinity. Moreover, $f_n\to f$, $(f_n+f_n')\to g$. Therefore $\{f,g\}\in I+T$ so $f\in dom(T)=dom(I+T)$.
		
	\end{proof}

	Let us summarize the known results of linear relations on a Hilbert space $\mathbb{H}$.\\
	
	\begin{theorem}
		Let $\mathbb{H}$ be a Hilbert space and $T$ a linear relation on $\mathbb{H}$. Then the following hold:
		\begin{enumerate}
			\item $(dom(T))^\perp = mul(T^*)$, $(ran(T))^\perp = ker(T^*)$ \cite{sandovici}
			\item $T^*$ is a closed linear relation \cite{arens}
			\item $T^{**}$ is the closure of $T$ so $T\subseteq T^{**}$. Moreover, if $T$ is closed then $T=T^{**}$. \cite{arens}
			\item The linear relation $T$ is single valued (in this case it is a graph of an operator) if and only if $(dom(T))^\perp=\{0\}$. \cite{arens}
			\item If $T$ is a symmetric linear relation on $\mathbb{H}$ and $dom(T)=\mathbb{H}$ then $T$ is a graph of a bounded self-adjoint operator. \cite{sandovici}
			\item If $T$ is a symmetric linear relation on $\mathbb{H}$ and $ran(T)=\mathbb{H}$ then $T$ is a self-adjoint linear relation on $\mathbb{H}$. \cite{sandovici}
			\item If $J:\mathbb{H}\times\mathbb{H}\to \mathbb{H}\times\mathbb{H}$, $J(\{f,g\}):=\{g,-f\}$, then $T^*=JT^\perp=(JT)^\perp$. \cite{sandovici,arens}
			\item If $T$ is a positive self-adjoint linear relation on $\mathbb{H}$ then  $ran(I+T)=\mathbb{H}$ where $I$ the graph of the identity operator. \cite{sandovici}
			\item Let $S$ and $T$ be two linear relations on $\mathbb{H}$. Then $S\subseteq T$ implies  $T^*\subseteq S^*$. \cite{arens}
			\item $(T^{-1})^*=(T^*)^{-1}$. \cite{arens}
		\end{enumerate}
	\end{theorem}
	
	\begin{theorem}
		Let $T$ and $S$ be two linear relations on $\mathbb{H}$. If the kernel and range of $T$ and $S$ are equal, then $S\subseteq T$ if and only if $S=T$. \cite{arens}
	\end{theorem}

	\begin{theorem} 
		(Generalization of von Neumann Theorem, \cite{sandovici}) Let $T$ be a closed linear relation on $\mathbb{H}$. Then the product $T^*T$ and  $TT^*$ are a positive self-adjoint linear relations on $\mathbb{H}$.
		
	\end{theorem}
	
	The next theorem is the reformulation of the general Stone-lemma.
	
	\begin{theorem}
		\label{new3}
		Let $\mathbb{H}$ be a Hilbert space and $T$ a linear relation on $\mathbb{H}$. If $ran(T\cap T^*)=\mathbb{H}$ then $T$ is self-adjoint.
	\end{theorem}
	\begin{proof}
		
		Let us consider $S=T\cap T^*$.
		
		\begin{enumerate}[label=(\alph*)]
			\item $T\subseteq T^*$ case:\\ 
			
			Let $\{g,g'\}\in T$. Then there exists $\{h,h'\}\in S$ such that $g'=h'$. Therefore $\{g-h,0\}\in T^*$ and there exists $\{k,k'\}\in S$ such that $k'=g-h$.\\
			
			 Moreover $\langle g-h,g-h\rangle=\langle k',g-h\rangle = \langle k,0\rangle $, so $h=g$. 
			 
			 Finally,  
			 $$\{g,g'\}=\{h,g'\}=\{h,h'\}\in S=T\cap T^*\subseteq T^*$$ 
			 so $\{g,g'\}\in T^*$,  therefore $T\subseteq T^*$.\\
			
			\item $T^*\subseteq T$ case:\\
			
			Let $\{g,g'\}\in T^*$. Then there exists $\{h,h'\}\in S$ such that $g'=h'$. Now we have $\{g-h,0\}\in T^*$ and there exists $\{k,k'\}\in S$ such that $k'=g-h$. \\
			
			Therefore $\langle g-h,g-h\rangle=\langle k',g-h\rangle = \langle k,0\rangle $ so $h=g$.\\
			
			 Finally 
			 $$\{g,g'\}=\{h,g'\}=\{h,h'\}\in S=T\cap T^*\subseteq T\text{,}$$
			  so $\{g,g'\}\in T$. As a consequence $T^*\subseteq T$.\\
			
			Summarising (a) and (b) we get that $T=T^*$.
		\end{enumerate}
	\end{proof}
	
	\begin{theorem}
		\label{new4}
		Let $\mathbb{H}$ be a Hilbert space and $T$ a positive linear relation on $\mathbb{H}$. Then the following are equivalent:
		\begin{enumerate}[label=(\arabic*)]
			\item $(dom(T))^\perp=0$ and $T^*=T$,
			\item There exists an operator such that $G(S)=T\cap T^*$ and \\
			$ran(I+S)=\mathbb{H}$.
		\end{enumerate}
	\end{theorem}
	\begin{proof}
		
		$(1)\to (2)$:
		
		$T=G(S)$ i.e. $S$ is a positive self-adjoint operator. As a consequence $ran(I+S)=\mathbb{H}$.
		
		$(2)\to (1)$:
		
		Now $S$ is a positive symmetric operator, therefore $(dom(S))^\perp=\{0\}$. Then $S$ is a positive self-adjoint operator.
		
		\begin{enumerate}[label=\alph*)]
			\item $g\in (dom(T)^\perp)$: \\
			There exists $h\in dom(S)\subseteq dom(T)$ such that $g=h+Sh$. 
			
			We know that $h\in dom(T)$ therefore 
			
			$$0=\langle g,h\rangle = \langle h,h\rangle + \langle Sh,h\rangle\text{.}$$ 
			
			We also know that $S$ is positive i.e.  $\langle Sh,h\rangle$ so $h=0$. \\
			
			Finally the linearity of $S$ implies that $g=0$ and as a consequence $T$ is densely defined.\\
			
			\item $\{0,g\}\in T^*$: \\
			$0=\langle h',0\rangle=\langle h,g\rangle$ for all $\{h,h'\}\in T$ therefore $g\in (dom(T))^\perp$. As a consequence $T^*$ is the graph of an operator. \\
			
			\item $g\in dom(T^*)$: \\
			
			There exists $h\in dom(S)$ such that $g+T^*g=h+Sh$. We know that $G(S)\subseteq T^*$ therefore $(g-h)+T^*(g-h)=0$.\\
			
			Then $(g-h)\in ker(I+T^*)=ker((I+T)^*)=(ran(I+T))^\perp\subseteq (ran(I+S))^\perp=\{0\}$. \\
			
			As a consequence $T^*g=Sh=Sg=Tg$ therefore $g=h\subseteq som(S)\subseteq dom(T)$. \\
			
			With this, we proved that $T^*\subseteq T$.\\
			
			\item $g\in dom(T)$: \\
			
			There exists $h\in dom(S)$ such that $g+Tg=h+Th$. 
			
			We know that $G(S)\subseteq T$ so $(g-h)+T(g-h)=0$.
			
			We also know that $T$ is positive so  $\langle T(g-h),g-h\rangle \geq 0$.
			
			  \begin{align*}
			  	\langle 0, g-h\rangle &= \langle (g-h)+T(g-h),g-h\rangle \\
			  	&= \langle g-h,g-h\rangle + \langle T(g-h),g-h\rangle 
			  \end{align*} 
			  
			  i.e. $g=h\subseteq dom(S)\subseteq dom(T^*)$\\
			
			So we proved that $T\subseteq T^*$. \\
		\end{enumerate}
		
		Summarizing the previous points we get that $T$ is a densely defined self-adjoint operator.
		
	\end{proof}

	\textit{Remark} The previous result not only gives us an other equivalent condition for a linear relation to be self-adjoint, but also shows that if the second condition holds, then $S$ will be the graph of an operator.

	\newpage	
	\bibliographystyle{plain}
	
\end{document}